\documentclass{elsarticle}

\usepackage[utf8]{inputenc}
\usepackage[T1]{fontenc}
\usepackage[english]{babel}

\makeatletter
\def\ps@pprintTitle{%
 \let\@oddhead\@empty
 \let\@evenhead\@empty
 \def\@oddfoot{}%
 \let\@evenfoot\@oddfoot}
\makeatother

\usepackage{booktabs}
 
\usepackage{amsfonts}
\usepackage{amsmath}
\usepackage{amssymb}
\usepackage{amsthm}
\usepackage{bbm}
\usepackage{xfrac}

\usepackage[colorlinks=false,hidelinks]{hyperref}
\usepackage{cleveref}

\usepackage{thmtools}
\usepackage{thm-restate}
\usepackage[dvipsnames]{xcolor}

\usepackage{ stmaryrd } 

\usepackage{enumitem}
\setenumerate{label=(\roman*)}

\newcommand{\HA}{\mathsf{HA}}

\newcommand{\Ord}{\mathrm{Ord}}
\renewcommand{\P}{\mathcal{P}}

\DeclareMathOperator{\ran}{ran}
\DeclareMathOperator{\dom}{dom}

\newcommand{\seq}[1]{\langle #1 \rangle}
\newcommand{\Seq}[2]{\langle #1 \, | \, #2 \rangle}
\newcommand{\set}[1]{\{ #1 \}}
\newcommand{\Set}[2]{\{ #1 \, | \, #2 \}}
\newcommand{\godel}[1]{\ulcorner #1 \urcorner}
\renewcommand{\phi}{\varphi}
\renewcommand{\epsilon}{\varepsilon}

\usepackage{scalerel,stackengine}
\stackMath
\newcommand\reallywidehat[1]{%
\savestack{\tmpbox}{\stretchto{%
  \scaleto{%
    \scalerel*[\widthof{\ensuremath{#1}}]{\kern-.6pt\bigwedge\kern-.6pt}%
    {\rule[-\textheight/2]{1ex}{\textheight}}
  }{\textheight}%
}{0.5ex}}%
\stackon[1pt]{#1}{\tmpbox}%
}

\newcommand{\M}{\mathcal{M}}

\DeclareMathOperator{\Ind}{Ind}

\DeclareMathOperator{\Fn}{Fn}

\newcommand{\CZF}{\mathsf{CZF}}

\newcommand{\IZF}{\mathsf{IZF}}

\newcommand{\ZFC}{\mathsf{ZFC}}
\newcommand{\ZF}{\mathsf{ZF}}

\newcommand{\PA}{\mathsf{PA}}

\newcommand{\AC}{\mathsf{AC}} 
\newcommand{\MP}{\mathsf{MP}} 

\newcommand{\CPC}{\mathbf{CPC}}
\newcommand{\CQC}{\mathbf{CQC}}
\newcommand{\IPC}{\mathbf{IPC}}
\newcommand{\IQC}{\mathbf{IQC}}

\newcommand{\Prop}{\mathsf{Prop}}

\newcommand{\sent}{\mathsf{sent}}

\newcommand{\Logic}{\mathbf{L}}
\newcommand{\QLogic}{\mathbf{QL}}

\declaretheorem[
name = Theorem,
]{theorem}
\declaretheorem[
name = Corollary,
sibling = theorem
]{corollary}
\declaretheorem[
name = Lemma,
sibling = theorem
]{lemma}
\declaretheorem[
name = Fact,
sibling = theorem
]{fact}
\declaretheorem[
name = Proposition,
sibling = theorem
]{proposition}

\declaretheorem[
name = Claim,
numbered = no
]{claim*}

\declaretheoremstyle[%
  spaceabove=-6pt,%
  spacebelow=6pt,%
  headfont=\normalfont\itshape,%
  postheadspace=1em,%
  qed=$\blacksquare$,%
  headpunct={.}
]{mystyle}

\declaretheorem[
name = Remark,
sibling = theorem,
style=definition
]{remark}

\declaretheorem[
name = Definition,
sibling = theorem,
style=definition
]{definition}

\declaretheorem[
name = Question,
sibling = theorem,
style=definition
]{question}

\declaretheorem[
name = Construction,
sibling = theorem,
style = definition
]{construction}

\usepackage{trimclip,adjustbox}

\newcommand{\IKP}{\mathsf{IKP}}

\newcommand{\Graph}{\mathsf{Graph}}

\DeclareMathOperator{\rank}{rank}

\begin{document}

\title{Logics of Intuitionistic {K}ripke-{P}latek Set Theory}

\author[Rosalie Iemhoff]{Rosalie Iemhoff\fnref{fn1}}
\address[Rosalie Iemhoff]{Department of Philosophy, Utrecht University, Janskerkhof 13, 3512 BL Utrecht, The Netherlands}
\ead{r.iemhoff@uu.nl}

\author[RPadd1,RPadd2]{Robert Passmann\corref{cor1}\fnref{fn2}}
\address[RPadd1]{Institute for Logic, Language and Computation, Faculty of Science, University of Amsterdam, P.O. Box 94242, 1090 GE Amsterdam, The Netherlands}
\address[RPadd2]{St John's College, University of Cambridge, Cambridge CB2 1TP, England}
\ead{r.passmann@uva.nl}

\cortext[cor1]{Corresponding author}
\fntext[fn1]{The first author is supported by the Netherlands Organisation for Scientific Research under grant 639.073.807.}
\fntext[fn2]{The second author was supported by a doctoral scholarship of the \emph{Studienstiftung des deutschen Volkes}, the \emph{Prins Bernhard Cultuurfonds}, and partially supported by the Marie Sk{\l}odowska-Curie fellowship REGPROP (706219) funded by the European Commission at the Universit{\"a}t Hamburg.}

\begin{abstract}
    We investigate the logical structure of intuitionistic Kripke-Platek set theory $\IKP$, and show that the first-order logic of $\IKP$ is intuitionistic first-order logic $\IQC$. 
\end{abstract}

\maketitle

\section{Introduction}

Any formal system is defined in essentially two crucial steps: First, choose a logic, and second, add some axioms for mathematical content. For example, Peano Arithmetic $\PA$ is defined by some arithmetical axioms on the basis of classical first-order logic. Heyting Arithmetic $\HA$ uses the same arithmetical axioms but is based on intuitionistic first-order logic. Similar situations arise in the context of set theories: Zermelo-Fraenkel Set Theory $\ZF$ is based on classical logic while its intuitionistic and constructive counterparts, $\IZF$ and $\CZF$, are based on intuitionistic logic. 

A feature of non-classical systems is that their logical strength \emph{can} increase with adding mathematical axioms. For example, Diaconescu \cite{Diaconescu1975} proved that the Axiom of Choice $\AC$ implies the Law of Excluded Middle in the context of intuitionistic $\IZF$ set theory. In other words, the system $\IZF + \AC$ is defined on the basis of intuitionistic logic but \emph{its logic is classical}. This illustrates the importance of determining the logic of any non-classical system of interest: By showing that an intuitionistic system indeed has intuitionistic logic, one verifies the conceptual requirement that the theory should be intuitionistic. The first result in this area was proved by De Jongh \cite{DeJonghUnpublished, DeJongh1970} who showed that the logic of Heyting Arithmetic $\HA$ is intuitionistic logic. This fact is now known as \emph{De Jongh's Theorem} (see \Cref{Definition: de Jongh property theorem} for more details).

Even though there is a rich literature on constructive set theories, there has not been much focus on the logics of these theories: Passmann \cite{Passmann2020} recently proved that the propositional logic of $\IZF$ is intuitionistic propositional logic $\IPC$. On the other hand, a result of H. Friedman and Ščedrov \cite{FriedmanScedrov1986} (see \Cref{Theorem: Friedman Scedrov}) implies that the first-order logic of intuitionistic set theories including full separation, such as $\IZF$, must be strictly stronger than intuitionistic first-order logic $\IQC$. These results show that $\IZF$ is logically well-behaved on the propositional level but less so on the level of predicate logic. 

What about other constructive set theories? Determining the first-order logic of $\CZF$, one of the most studied constructive set theories, is still an open problem. Another natural constructive set theory, that has been studied in the literature, is intuitionistic Kripke-Platek set theory $\IKP$. Lubarsky \cite{Lubarsky2002} introduced $\IKP$ to investigate intuitionistic admissibility theory in the tradition of Barwise \cite{Barwise1975}. In this article, we show that $\IKP$ is a logically very well-behaved theory as the following consequences of our more general results illustrate:
\begin{enumerate}
    \item the propositional logic of $\IKP$ is intuitionistic propositional logic $\IPC$ (see \Cref{Corollary: Propositional Logic IKP}),
    \item the relative first-order logic of $\IKP$ is intuitionistic first-order logic $\IQC$ (see \Cref{Corollary: Relative First-Order Logic IKP}), 
    \item the first-order logic of $\IKP$ is intuitionistic first-order logic $\IQC$ (see \Cref{Corollary: First-Order Logic IKP}), and,
    \item the first-order logic with equality of $\IKP$ is strictly stronger than intuitionistic first-order logic with equality $\IQC^=$ (see \Cref{Corollary: First-order logic with equality of IKP is stronger}).
\end{enumerate}

An important byproduct of our work is a study of the possibilities and limits of Kripke models whose domains are classical models of set theory: The common Kripke model constructions for intuitionistic or constructive set theories, such as $\CZF$ or $\IZF$, that are stronger than $\IKP$, usually involve complex recursive constructions (see, for example, \cite{Lubarsky2018b}). We will expose a failure of the exponentiation axiom showing that these more complex constructions are necessary to obtain models of many stronger theories (see \Cref{Subsection: A failure of exponentiation}).

This article is organised as follows. In \Cref{Section: Logics and Their Kripke Semantics}, we will lay out the necessary preliminaries concerning Kripke semantics for propositional and first-order logics. \Cref{Section: IKP and Its Kripke Semantics} provides an analysis of a certain Kripke model construction for $\IKP$. In \Cref{Section: The Logical Structure of IKP} we will analyse the logical structure of $\IKP$ and prove several \emph{De Jongh Theorems} for propositional, relative first-order and first-order logics. We close with some questions and directions for further research.

\section{Logics and Their Kripke Semantics}
\label{Section: Logics and Their Kripke Semantics}

As usual, we denote intuitionistic propositional logic by $\IPC$ and intuitionistic first-order logic by $\IQC$. The classical counterparts of these logics are called $\CPC$ and $\CQC$, respectively.  We will generally identify each logic with the set of its consequences. A logic $J$ is called intermediate if $\IPC \subseteq J \subseteq \CPC$ in case $J$ is propositional logic, or $\IQC \subseteq J \subseteq \CQC$ in case $J$ is a first-order logic. We assume intuitionistic first-order logic $\IQC$ to be formulated in a language without equality. Intuitionistic first-order logic \emph{with equality} will be denoted by $\IQC^=$.

A \emph{Kripke frame} $(K,\leq)$ is a set $K$ equipped with a partial order $\leq$. A \emph{Kripke model for $\IPC$} is a triple $(K,\leq,V)$ such that $(K,\leq)$ is a Kripke frame and $V: \Prop \to \P(K)$ a valuation that is persistent, i.e., if $w \in V(p)$ and $w \leq v$, then $v \in V(p)$. We can then define, by induction on propositional formulas, the forcing relation for propositional logic at a node $v \in K$ in the following way for a Kripke model $M$ for $\IPC$:

\begin{enumerate}[label=(\arabic*), series=ForcingDef]
    \item $M,v \Vdash p$ if and only if $v \in V(p)$,
    \item $M,v \Vdash \phi \wedge \psi$ if and only if $K,V,v \Vdash \phi \text{ and } K,V,v \Vdash \psi$,
    \item $M,v \Vdash \phi \vee \psi$ if and only if $K,V,v \Vdash \phi \text{ or } K,V,v \Vdash \psi$,
    \item $M,v \Vdash \phi \rightarrow \psi$ if and only if for all $w \geq v$, $K,V,w \Vdash \phi$ implies $K,V,w \Vdash \psi$,
    \item $M,v \Vdash \bot$ holds never.
\end{enumerate}
    
We write $v \Vdash \phi$ instead of $K, V, v \Vdash \phi$ if the Kripke frame and the valuation are clear from the context. We will write $K, V \Vdash \phi$ if $K,V,v \Vdash \phi$ holds for all $v \in K$. A formula $\phi$ is \emph{valid in $K$} if $K, V, v \Vdash \phi$ holds for all valuations $V$ on $K$ and $v \in K$, and $\phi$ is \emph{valid} if it is valid in every Kripke frame $K$.
    
We can now define the propositional logic of a Kripke frame and of a class of Kripke frames.

\begin{definition}
    If $(K,\leq)$ is a Kripke frame, we define the \emph{propositional logic $\Logic(K,\leq)$} to be the set of all propositional formulas that are valid in $K$. For a class $\mathcal{K}$ of Kripke frames, we define the \emph{propositional logic $\Logic(\mathcal{K})$} to be the set of all propositional formulas that are valid in all Kripke frames $(K,\leq)$ in $\mathcal{K}$. Given an intermediate propositional logic $\mathbf{J}$, we say that \emph{$\mathcal{K}$ characterises $\mathbf{J}$} if $\Logic(\mathcal{K}) = \mathbf{J}$.
\end{definition}

A \emph{Kripke model for $\IQC$} is a triple $(K,\leq,D,V)$ where $(K,\leq)$ is a Kripke frame, $D_v$ a set for each $v \in K$ such that $D_v \subseteq D_w$ for $v \leq w$, and $V$ a function such that:
    \begin{enumerate}
        \item if $p$ is a propositional letter, then $V(p) \subseteq K$ such that if $v \in V(p)$ and $v \leq w$, then $w \in V(p)$,
        \item if $R$ is an $n$-ary relation symbol of the language of $\IQC$, then $V(R) = \Set{R_v}{v \in K}$ such that $R_v \subseteq D_v^n$ and $R_v \subseteq R_w$ for $v \leq w$, and, 
        \item if $f$ is an $n$-ary function symbol of the language of $\IQC$, then $V(f) = \Set{f_v}{v \in K}$ such that $f_v$ is a function $D_v^n \to D_v$ such that $\Graph(f_v) \subseteq \Graph(f_w)$ for $v \leq w$.
    \end{enumerate}
    
We now extend the conditions of the forcing relation for $\IPC$ to Kripke models $M$ for $\IQC$ in the following way:
\begin{enumerate}[resume*=ForcingDef]
    \item $M,v \Vdash R(x_0,\dots,x_{n-1})$ if and only if $(x_0,\dots,x_{n-1}) \in R_v$,
    \item $M,v \Vdash f(x_0,\dots,x_{n-1}) = y$ if and only if $f_v(x_0,\dots,x_{n-1}) = y$,
    \item $M,v \Vdash \exists x \ \phi(x)$ if and only if there is some $x \in D_v$ such that $K,V,v \Vdash \phi(x)$, and,
    \item $M,v \Vdash \forall x \ \phi(x)$ if and only if for all $w \geq v$ and $x \in D_w$ it holds that $K,V,w \Vdash \phi(x)$.
\end{enumerate}

We can further extend these definitions to Kripke models for $\IQC^=$ by interpreting equality as a congruence relation $\sim_v$ at every node $v \in K$, and stipulate that:
\begin{enumerate}[resume*=ForcingDef]
    \item $M,v \Vdash x = y$ if and only if $x \sim_v y$.
\end{enumerate}

We define the validity of formulas in frames and classes of frames just as in the case of propositional logic. Now, we can define the first-order logic of a Kripke frame and of a class of Kripke frames.

\begin{definition}
    If $(K,\leq)$ is a Kripke frame, then the \emph{first-order logic $\QLogic(K,\leq)$} is defined to be the set of all first-order formulas that are valid in $K$. For a class $\mathcal{K}$ of Kripke frames, we define the \emph{first-order logic $\QLogic(\mathcal{K})$} to be the set of all first-order formulas that are valid in all Kripke frames $(K,\leq)$ in $\mathcal{K}$. Given an intermediate first-order logic $\mathbf{J}$, we say that \emph{$\mathcal{K}$ characterises $\mathbf{J}$} if $\QLogic(\mathcal{K}) = \mathbf{J}$. 
    
    Similarly, we define $\QLogic^=(\mathcal(K,\leq))$ and $\QLogic^=(\mathcal{K})$ as the set of all first-order formulas in the language of equality that are valid in the respective frame or class of frames.
\end{definition}

We will sometimes write $\Logic(K)$ for $\Logic(K,\leq)$, $\QLogic(K)$ for $\QLogic(K,\leq)$, and $\QLogic^=(K)$ for $\QLogic^=(K,\leq)$. The next result is proved by induction on the complexity of formulas; it shows that persistence of the propositional variables transfers to all formulas.

\begin{proposition}
    Let $M$ be a Kripke model for $\IPC$, $\IQC$ or $\IQC^=$, $v \in K$ and $\phi$ be a propositional formula such that $M,v \Vdash \phi$ holds. Then $M, w \Vdash \phi$ holds for all $w \geq v$. \qed
\end{proposition}

\begin{theorem}
    \label{Theorem:IQC true in Kripke models}
    A propositional formula $\phi$ is derivable in $\IPC$ if and only if it is valid in all Kripke models for $\IPC$. In particular, a propositional formula $\phi$ is derivable in $\IPC$ if and only if it is valid in all finite Kripke models for $\IPC$. A formula $\phi$ of first-order logic is derivable in $\IQC$ if and only if it is valid in all Kripke models for $\IQC$. Finally, a formula $\phi$ of first-order logic with equality is derivable in $\IQC^=$ if and only if it is valid in all Kripke models for $\IQC^=$.
\end{theorem}

A detailed proof of \Cref{Theorem:IQC true in Kripke models} can be found in the literature (e.g., \cite[Chapter 2, Theorem 6.6]{TroelstraVanDalen}). We use this opportunity to emphasise that we are using $\IQC$ to denote first-order intuitionistic logic in a language without equality, and $\IQC^=$ to denote first-order intuitionistic logic with equality (see \cite[Chapter 2]{TroelstraVanDalen} for a discussion of various versions of intuitionistic first-order logic \emph{with} and \emph{without} equality).

We will later need the following result on Kripke frames for $\IQC$ (\emph{without} equality). 

\begin{definition} 
    We say that a Kripke model $M = (K,\leq,D,V)$ is \emph{countable} if $K$ is countable and $D_v$ is countable for every $v \in K$.  A Kripke model $M = (K,\leq,D,V)$ has \emph{countably increasing domains} if for every $v, w \in K$ such that $v < w$, we have that $D_w \setminus D_v$ is a countably infinite set.
\end{definition}

\begin{lemma}
    \label{Lemma: Kripke Model Extension of Domains}
    Let $M = (K,\leq,D,V)$ be a countable Kripke model for intuitionistic first-order logic. Then there is a model $M' = (K,\leq,D',V')$ with countably increasing domains and a family of maps $f_v: D_v \to D'_v$ such that $M, v \Vdash \phi(\bar x)$ if and only if $M', v \Vdash \phi(f_v(\bar x))$ holds for every $v \in K$. Further, if $M$ is countable, then so is $M'$.
\end{lemma}
\begin{proof}
    As $M$ is countable, the Kripke frame $(K,\leq)$ will be countable. So let $\Seq{v_i}{i < \omega}$ be a bijective enumeration of all nodes of $K$. Let $M_0 = M$. Given $M_n = (K,\leq,D^n,V^n)$, define $M_{n+1}$ as follows: Take a countable set $X_n$ such that $X_n \cap \bigcup_{v \in K}{D^n_v} = \emptyset$. Now let $D^n_w = D^n_w$ if $w \not \geq v_n$, and $D^n_w = D^n_w \cup X_n$ if $w \geq v_n$. Extend the valuation $V^n$ of $M_n$ to the extended domains as follows: Pick an arbitrary element $y_n \in D^n_{v_n}$ and copy the valuation of $y_n$ for every $x \in X_n$ at every $w \geq v_n$ (i.e. such that $v \Vdash P(x,\bar z)$ if and only if $v \Vdash P(y_n,\bar z)$).
    
    Finally, take $M' = (K, \leq, D', V')$ where $D'_v = \bigcup_{n < \omega} D^n_v$ and $V'_v = \bigcup_{n < \omega} V^n_v$. Clearly $M'$ is countable. Further define $f: \bigcup_{v \in K} D'_v \to \bigcup_{v \in K} D_v$ by stipulating that $f(x) = x$ if $x \in D_v$, and $f(x) = y_n$ if $x \in X_n$. An easy induction now shows that the desired statement holds (note that the language of $\IQC$ does not contain equality). 
\end{proof}

\section{IKP and Its Kripke Semantics}
\label{Section: IKP and Its Kripke Semantics}

In this section, we will introduce intuitionistic Kripke-Platek set theory, and Kripke semantics for this theory.

\subsection{Intuitionistic Kripke-Platek Set Theory}

We will list the relevant axioms and axiom schemes. The language $\mathcal{L}_\in$ of set theory extends the logical language with binary relation symbols $\in$ and $=$ denoting set membership and equality, respectively. As usual, the bounded quantifiers $\forall x \in a \ \phi(x)$ and $\exists x \in a \ \phi(x)$ are abbreviations for $\forall x (x \in a \rightarrow \phi(x))$ and $\exists x (x \in a \wedge \phi(x))$, respectively.
    \begin{align*} \allowdisplaybreaks
        & \exists a \ \forall x \in a \ \bot \tag{Empty Set} \\
        & \forall a \forall b \exists y \forall x (x \in y \leftrightarrow (x = a \vee x = b)) \tag{Pairing} \\
        & \forall a \exists y \forall x (x \in y \leftrightarrow \exists u (u \in a \wedge x \in u)) \tag{Union} \\
        & \forall a \forall b (\forall x(x \in a \leftrightarrow x \in b) \rightarrow a = b) \tag{Extensionality} \\
        & \exists x (\emptyset \in x \wedge (\forall y \ y \in x \rightarrow y \cup \set{y} \in x) \wedge \tag{Infinity} \\ & (\forall y \ y \in x \rightarrow (y = \emptyset \vee \exists z \in y \ y = z \cup \set{z}))) \\
        & (\forall a (\forall x \in a \ \phi(x) \rightarrow \phi(a))) \rightarrow \forall a \phi(a)  \tag{Set Induction}
    \end{align*}
Moreover, we have the axiom schemes of $\Delta_0$-separation and $\Delta_0$-collection, where $\phi$ ranges over the bounded formulas:
    \begin{align*}
        & \forall a \exists y \forall x (x \in y \leftrightarrow x \in a \wedge \phi(x)) \hspace{1em}  \text{($\phi$ is a $\Delta_0$-formula)} \tag{$\Delta_0$-Separation} \\
        & \forall a( \forall x\in a \exists y \phi(x,y) \rightarrow \exists b\forall x\in a \exists y\in b \phi(x,y) )
        \hspace{1em}  \text{($\phi$ is a $\Delta_0$-formula)} \tag{$\Delta_0$-Collection}
    \end{align*}
Sometimes, these schemes are also referred to as \emph{bounded separation} and \emph{bounded collection}, respectively. Removing the restriction to $\Delta_0$-formulas, we obtain the usual schemes of \emph{separation} and \emph{collection}.
    
\begin{definition}
    The theory $\IKP$ of \emph{intuitionistic Kripke-Platek set theory $\IKP$} consists of the axioms and rules of intuitionistic first-order logic for the language $\mathcal{L}_\in$ extended by the axioms and axiom schemes of empty set, pairing, union, extensionality, infinity, set induction, $\Delta_0$-separation, and $\Delta_0$-collection.    
\end{definition}

$\IKP$ was first introduced and studied by Lubarsky \cite{Lubarsky2002}. Denote by $\IKP^+$ the theory obtained by adding the schemes of bounded strong collection and set-bounded subset collection to $\IKP$.

For reference, we also introduce the well-known theories of $\CZF$ and $\IZF$.
In the following strong infinity axiom, $\Ind(a)$ is the formula denoting that $a$ is an inductive set: $\Ind(a)$ abbreviates $\emptyset \in a \wedge \forall x \in a \exists y \in a \ y = \set{x}$.
    \begin{align*}
        & \exists a (\Ind(a) \wedge \forall b (\Ind(b) \rightarrow \forall x \in a (x \in b))) \tag{Strong Infinity}
    \end{align*}    
Finally, we have the schemes of strong collection and subset collection for all formulas $\phi(x,y)$ and $\psi(x,y,u)$, respectively.
    \begin{align*}
        & \forall a (\forall x \in a \exists y \ \phi(x,y) \rightarrow \tag{$\text{Strong Collection}$} \\
        & \hspace{2em} \exists b (\forall x \in a \exists y \in b \ \phi(x,y) \wedge \forall y \in b \exists x \in a \ \phi(x,y)) )  \\ 
        & \forall a \forall b \exists c \forall u (\forall x \in a \exists y \in b \ \psi(x,y,u) \rightarrow  \tag{$\text{Subset Collection}$} \\
        & \hspace{2em} \exists d \in c(\forall x \in a \exists y \in d \ \psi(x,y,u) \wedge \forall y \in d \exists x \in a \ \psi(x,y,u)))
    \end{align*}
    
The axiom scheme obtained from strong collection when restricting $\phi$ to range over $\Delta_0$-formulas only will be called  \emph{Bounded Strong Collection}. Similarly, we obtain the axiom scheme of \emph{Set-bounded Subset Collection} from the axiom scheme of subset collection when restricting $\psi$ to $\Delta_0$-formulas such that $z$ is set-bounded in $\psi$ (i.e., it is possible to intuitionistically derive $z \in t$ for some term $t$ that appears in $\psi$ from $\psi(x,y,z)$). 

We also need the power set axiom. 
\begin{align*}
    & \forall a \exists y \forall z (z \in y \leftrightarrow z \subseteq a)
        \tag{Power Set}
\end{align*} 
    
\begin{definition} \label{Definition:CZF}
    The theory $\CZF$ of \emph{constructive Zermelo-Fraenkel set theory} consists of the axioms and rules of intuitionistic first-order logic for the language $\mathcal{L}_\in$ extended by the axioms of extensionality, empty set, pairing, union and strong infinity as well as the axiom schemes of set induction, bounded separation, strong collection and subset collection.
\end{definition}

In the statement of the following axiom of exponentiation, $f:x \to y$ is an abbreviation for the $\Delta_0$-formula $\phi(f,x,y)$ stating that $f$ is a function from $x$ to $y$. 
\begin{equation*}
    \forall x \ \forall y \ \exists z \ \forall f (f \in z \ \leftrightarrow \ f: x \to y)
    \tag{Exponentiation, $\mathsf{Exp}$}
\end{equation*}
The axiom of exponentiation is a constructive consequence of the axiom of subset collection over $\CZF$ (cf. \cite[Theorem 5.1.2]{AczelRathjen2010}). Hence, a failure of exponentiation implies a failure of subset collection. We will see in \Cref{Subsection: A failure of exponentiation} that the Kripke models with classical domains do not satisfy the axiom of exponentiation in general, and therefore, they cannot satisfy full $\CZF$.

\begin{definition}
    The theory $\IZF$ of \emph{intuitionistic Zermelo-Fraenkel set theory} consists of the axioms and rules of intuitionistic first-order logic for the language $\mathcal{L}_\in$ extended by the axioms and axiom schemes of extensionality, pairing, union, empty set, strong infinity, separation, collection, set induction, and powerset. 
\end{definition}

\subsection{Kripke Models for the Language of Set Theory}

By extending the Kripke models introduced above, we can obtain models for intuitionistic first-order logic. Instead of developing this theory in full generality, we will focus on the subcase of \emph{Kripke models for set theory}.

\begin{definition}
    A \emph{Kripke model $(K,\leq,D,e)$ for set theory} is a Kripke frame $(K,\leq)$ for $\IPC$ with a collection of domains $D = \Set{D_v}{v \in K}$ and a collection of set-membership relations $e = \Set{e_v}{v \in K}$, such that the following hold:
    \begin{enumerate}
        \item $e_v$ is a binary relation on $D_v$ for every $v \in K$, and,
        \item $D_v \subseteq D_w$ and $e_v \subseteq e_w$ for all $w \geq v \in K$.
    \end{enumerate}
\end{definition}

Examples of Kripke models for set theory are not only the Kripke models with classical domains that we will introduce in \Cref{Section: Kripke Models with Classical Domains}, but also the Kripke models introduced by Lubarsky \cite{Lubarsky2005,Lubarsky2018a}, by Diener and Lubarsky \cite{LubarskyDiener2014}  and by Lubarsky and Rathjen \cite{LubarskyRathjen2008}; recently Passmann \cite{Passmann2020} introduced the so-called \emph{blended Kripke models for set theory} to prove de Jongh's theorem for $\IZF$ and $\CZF$.

We can now extend the forcing relation to Kripke models for set theory to interpret the language of set theory $\mathcal{L}_\in$. For the following definition, we tacitly enrich the language of set theory with constant symbols for every element of the domains of the Kripke model at hand.

\begin{definition}
    Let $(K,\leq,D,e)$ be a Kripke model for set theory. We define, by induction on $\mathcal{L}_\in$-formulas, the forcing relation at every node of a Kripke frame in the following way, where $\phi$ and $\psi$ are formulas with all free variables shown, and $\bar y = y_0,\dots,y_{n-1}$ are elements of $D_v$ for the node $v$ considered on the left side: 
    \begin{enumerate}
        \item $(K,\leq,D,e),v \Vdash a \in b$ if and only if $(a,b) \in e_v$,
        \item $(K,\leq,D,e),v \Vdash a = b$ if and only if $a = b$,
        \item $(K,\leq,D,e),v \Vdash \exists x \phi(x,\bar y)$ if and only if there is some $a \in D_v$ \\ with $(K,\leq,D,e),v \Vdash \phi(a, \bar y)$,
        \item $(K,\leq,D,e),v \Vdash \forall x \phi(x,\bar y)$ if and only if for all $w \geq v$ and $a \in D_w$ \\ we have $(K,\leq,D,e),w \Vdash \phi(a,\bar y)$.
    \end{enumerate}
    The cases for $\rightarrow$, $\wedge$, $\vee$ and $\bot$ are analogous to the ones in the above definition of the forcing relation for Kripke models for $\IPC$. We will write $v \Vdash \phi$ (or $K, v \Vdash \phi$) instead of $(K,\leq,D,e), v \Vdash \phi$ if the Kripke model is clear from the context. An $\mathcal{L}_\in$-formula $\phi$ is \emph{valid in $K$} if $v \Vdash \phi$ holds for all $v \in K$, and $\phi$ is \emph{valid} if it is valid in every Kripke frame $K$. Finally, we will call $(K,\leq)$ the \emph{underlying Kripke frame} of $(K,\leq,D,e)$.
\end{definition}

Persistence also holds in Kripke models for set theory.

\begin{proposition}
    Let $(K,\leq,V)$ be a Kripke model for set theory, $v \in K$ and $\phi$ be a formula in the language of set theory such that $K,v \Vdash \phi$ holds. Then $K, w \Vdash \phi$ holds for all $w \geq v$. \qed
\end{proposition}

\begin{remark}
    We have now introduced four kinds of Kripke models: for $\IPC$, for $\IQC$, for $\IQC^=$, and for set theory. The reader might have noticed that Kripke models for set theory are just a special instance of the Kripke models for $\IQC^=$ where equality is interpreted as actual equality on the domains. Kripke models for $\IQC^=$ do in general not interpret equality this way and only require an equivalence relation, and Kripke models for $\IQC$ do not have equality at all. Using this distinction, we are making explicit when we talk about Kripke models for certain \emph{logics} and when we are talking about Kripke models for certain \emph{set theories}.
\end{remark}

\subsection{Kripke Models with Classical Domains}
\label{Section: Kripke Models with Classical Domains}

The idea is to obtain models of set theory by assigning classical models of $\ZF$ set theory to every node of a Kripke frame. We will first introduce Kripke models with classical domains and explain some of their basic properties. Afterwards, we will indicate their limitations in modelling strong set theories by exhibiting a failure of the exponentiation axiom.

\subsubsection{Definitions and Basic Properties}

We will closely follow the presentation of Iemhoff \cite{Iemhoff2010} but give up on some generality that is not needed here. We will start by giving a condition for when an assignment of models to nodes is suitable for our purposes.

\begin{definition}
    Let $(K,\leq)$ be a Kripke frame. An assignment $\M: K \to V$ of transitive models of $\ZF$ set theory to nodes of $K$ is called \emph{sound for $K$} if for all nodes $v, w \in K$ with $v \leq w$ we have that $\M(v) \subseteq \M(w)$, $\M(v) \vDash x \in y$ implies $\M(w) \vDash x \in y$, and $\M(v) \vDash x = y$ implies $\M(w) \vDash x = y$.
\end{definition}

For convenience, we will write $\M_v$ for $\M(v)$. Of course, this could be readily generalised to homomorphisms of models of set theory that are not necessarily inclusions, but we will not need this level of generality here.

\begin{definition}
    Given a Kripke frame $(K,\leq)$ and a sound assignment $\M:K \to V$, we define the \emph{Kripke model with classical domains} $K(\M)$ to be the Kripke model for set theory $(K,\leq,\M,e)$ where $e_v = {\in} \upharpoonright (\M_v \times \M_v)$.
\end{definition}

Persistence for Kripke models with classical domains is a special case of persistence for Kripke models for set theory.

\begin{proposition}\label{Proposition:Iemhoff persistence}
    If $K(\M)$ is a Kripke model with classical domains with nodes $v, w \in K$ such that $v \leq w$, then for all formulas $\phi$, $K(\M), v \Vdash \phi$ implies $K(\M), w \Vdash \phi$. \qed
\end{proposition}

We will now analyse the set theory satisfied by these models.

\begin{definition}
    We say that a set-theoretic formula $\phi(x_0,\dots,x_{n-1})$ is \emph{evaluated locally} if for all Kripke models with classical domains $K(\M)$, where $\M$ is a sound assignment, we have $K(\M), v \Vdash \phi(a_0,\dots,a_{n-1})$ if and only if $\M_v \vDash \phi(a_0,\dots,a_{n-1})$ for all $a_0,\dots,a_{n-1} \in \M_v$.
\end{definition}

\begin{proposition} \label{Proposition:Delta_0 locally}
    If $\phi$ is a $\Delta_0$-formula, then $\phi$ is evaluated locally.
\end{proposition}
\begin{proof}
    This statement can be shown by actually proving a stronger statement by induction on $\Delta_0$-formulas, simultaneously for all $v \in K$. Namely, we can show that for all $w \geq v$ it holds that $w \Vdash \phi(a_0,\dots,a_n)$ if and only if $\M_v \vDash \phi(a_0,\dots,a_n)$. To prove the case of the bounded universal quantifier and the case of implication, we need that the quantifier is outside in the sense that our induction hypothesis will be:
    \begin{equation*}
        \forall w \geq v (w \Vdash \phi(a_0,\dots,a_n) \iff \M_v \vDash \phi(a_0,\dots,a_n)).
    \end{equation*}
    With this setup, the induction follows straightforwardly. 
\end{proof}

\begin{theorem}[Iemhoff, {\cite[Corollary 4]{Iemhoff2010}}] \label{Theorem:Iemhoff satisfies BCZF}
    Let $K(\M)$ be a Kripke model with classical domains. Then $K(\M) \Vdash \IKP^+$.
\end{theorem}

Recall that Markov's principle $\mathsf{MP}$ is formulated in the context of set theory as follows: $$\forall \alpha : \mathbb{N} \to 2 \ (\neg \forall n \in \mathbb{N} \ \alpha(n) = 0 \rightarrow \exists n \in \mathbb{N} \ \alpha(n) = 1)$$

\begin{proposition}
    Let $K(\M)$ be a Kripke model with classical domains. Then $K(\M) \Vdash \MP$.
\end{proposition}
\begin{proof}
    Let $v \in K$ and $\alpha \in \M_v$ be given such that $v \Vdash \text{``$\alpha$ is a function $\alpha \to 2$''}$. By \Cref{Proposition:Delta_0 locally}, we know that $\alpha$ is such a function also in the classical model $\M_v$. Further observe that $\neg \forall n \in \mathbb{N} \ \alpha(n) = 0 \rightarrow \exists n \in \mathbb{N} \ \alpha(n) = 1$ is a $\Delta_0$-formula and therefore evaluated locally by \Cref{Proposition:Delta_0 locally}. Now this statement is clearly true of $\alpha$ because $\M_v$ is a classical model of $\ZF$. 
\end{proof}

Extended Church's Thesis $\mathsf{ECT}$ does not hold.\footnote{This follows because under $\mathsf{MP}$ and $\mathsf{ECT}$ all functions $f: \mathbb{R} \to \mathbb{R}$ are continuous (see \cite[Theorem 16.0.23]{AczelRathjen2010}) but that is in general not the case here.} Let us conclude this section with the following curious observation.

\begin{proposition}
    If $K(\M)$ is a Kripke model with classical domains such that every $\M_v$ is a model of the axiom of choice, then the axiom of choice holds in $K(\M)$.
\end{proposition}
\begin{proof}
    Recall that the axiom of choice is the following statement:
    \begin{equation*}
    \forall a ( (\forall x \in a \forall y \in a \ (x \neq y \rightarrow x \cap y = \emptyset) ) \rightarrow \exists b \forall x \in a \exists ! z \in b \ z \in x  ).
    \tag{$\mathsf{AC}$}
    \end{equation*}
    Let $v \in K$ and $a \in \M_v$ such that $v \Vdash \forall x \in a \forall y \in a \ (x \neq y \rightarrow x \cap y = \emptyset)$. This is a $\Delta_0$-formula,  so \Cref{Proposition:Delta_0 locally} yields that $\M_v \vDash \forall x \in a \forall y \in a \ (x \neq y \rightarrow x \cap y = \emptyset)$. As $\M_v \vDash \mathsf{AC}$, there is some $b \in \M_v$ such that $\M_v \vDash \forall x \in a \exists ! z \in b \ z \in x$. Again, this is a $\Delta_0$-formula, so it holds that $v \Vdash \forall x \in a \exists ! z \in b \ z \in x$. As $b \in \M_v$, we have $v \Vdash \exists b \forall x \in a \exists ! z \in b \ z \in x$. But this shows that $v \Vdash \mathsf{AC}$.
\end{proof}

As $\IKP^+$ contains the bounded separation axiom, it follows that $\mathsf{AC}$ implies the law of excluded middle for bounded formulas in the models of the proposition (see \cite[Chapter 10.1]{AczelRathjen2010}). We summarise the results of this section in the following corollary.

\begin{corollary}
    If $K(\mathcal{M})$ is a Kripke model with classical domains such that every $\M_v$ is a model of the axiom of choice, then $K(\mathcal{M}) \Vdash \IKP^+ + \MP + \AC$. \qed
\end{corollary}

\subsubsection{A Failure of Exponentiation}
\label{Subsection: A failure of exponentiation}

In this section, we will exhibit a failure of the axiom of exponentiation in particular Kripke models with classical domains.\footnote{The results in this section are based on the third chapter of the second author's master's thesis \cite{Passmann2018}, supervised by Benedikt Löwe at the University of Amsterdam.}

\begin{proposition}
    \label{Proposition:Iemhoff no Exp}
    Let $K(\M)$ be a Kripke model with classical domains such that there are $v, w \in K$ with $v < w$. If $a, b \in \M_v$ and $g: a \to b$ is a function contained in $\M_w$ but not in $\M_v$, then $K(\M) \not \Vdash \mathsf{Exp}$.
\end{proposition}
\begin{proof}    
    Assume, towards a contradiction, that $K(\M) \Vdash \mathsf{Exp}$. Further, assume that $a, b \in \M_v$ and $g: a \to b$ is a function contained in $\M_w$ but not in $\M_v$. Then, 
    $$
    K(\M), v \Vdash \forall x \ \forall y \ \exists z \ \forall f (f \in z \ \leftrightarrow \ f: x \to y),
    $$ 
    and by the definition of our semantics this just means that there is some $c \in \M_v$ such that
    $
    K(\M), v \Vdash \forall f (f \in c \ \leftrightarrow \ f: a \to b).
    $
    By the semantics of universal quantification, this means that 
    $
    K(\M), w \Vdash g \in c \ \leftrightarrow \ g: a \to b.
    $
    Since $g$ is indeed a function from $a \to b$, it follows that
    $
    K(\M), w \Vdash g \in c.
    $
    As $c$ is a member of $\M_v$ by assumption, we have $g \in c \in \M_v$. Hence, by transitivity, $g \in \M_v$. But this is a contradiction to our assumption that $g$ is not contained in $\M_v$.
\end{proof}

Of course, when adding a generic filter for a non-trivial forcing notion, we always add such a function, namely the characteristic function of the generic filter. Therefore, \Cref{Proposition:Iemhoff no Exp} yields:

\begin{corollary}
    \label{Corollary:Iemhoff not CZF}
    Let $K(\M)$ be a Kripke model with classical domains. If there are nodes $v < w \in K$ such that $\M_w$ is a non-trivial generic extension of $\M_v$ (i.e., $\M_w = \M_v[G]$ for some generic $G \notin \M_v$), then it is not a model of $\CZF$. \qed
\end{corollary}

In Kripke semantics for intuitionistic logic, $K(\M) \Vdash \neg \phi$ is strictly stronger than $K(\M) \not \Vdash \phi$. The above results give an instance of the latter (a so-called weak counterexample), now we will provide an example of the former (a strong counterexample).

\begin{proposition}
    There is a Kripke model with classical domains $K(\M)$ that forces the negation of the exponentiation axiom, i.e., $K(\M) \Vdash \neg \mathsf{Exp}$.
\end{proposition}
\begin{proof}
    Consider the Kripke frame $K = (\omega,<)$ where $<$ is the standard ordering of the natural numbers. Construct the assignment $\M$ as follows: Choose $\M_0$ to be any countable and transitive model of $\ZFC$. If $\M_i$ is constructed, let $\M_{i+1} = \M_i[G_i]$ where $G_i$ is generic for Cohen forcing over $\M_i$ (actually, every non-trivial forcing notion does the job). Clearly, $\M$ is a sound assignment of models of set theory. Now, we want to show that for every $i \in \omega$ we have that $i \Vdash \neg \mathsf{Exp}$, i.e., for all $j \geq i$ we need to show that $j \Vdash \mathsf{Exp}$ implies $j \Vdash \bot$. This, however, is done exactly as in the proof of \Cref{Proposition:Iemhoff no Exp}, where the witnesses are the characteristic functions $\chi_{G_i}$ of the generic filters $G_i$.
\end{proof}

\subsection{Classical Domains and the Constructible Universe}

We define the relativisation $\phi \mapsto \phi^\mathrm{L}$ of a formula of set theory to the constructible universe $\mathrm{L}$ in the usual way. Note, however, that in our setting the evaluation of universal quantifiers and implications is in general not local (in contrast to classical models of set theory). Nevertheless, we will now show that---under mild assumptions---statements about the constructible universe can be evaluated locally. The following is a well-known fact.

\begin{fact}[{\cite[Lemma 13.14]{Jech2003}}]
    \label{Prop:L Sigma_1}
    There is a $\Sigma_1$-formula $\phi(x)$ such that in any model $M \vDash \ZFC$, we have $M \vDash \phi(x) \leftrightarrow x \in \mathrm{L}$. 
\end{fact}

From now on, let `$x \in \mathrm{L}$' be an abbreviation for $\phi(x)$, where $\phi$ is the $\Sigma_1$-formula from \Cref{Prop:L Sigma_1}.

\begin{proposition}\label{Proposition: L locally}
    Let $K$ be a Kripke frame and $\M$ a sound assignment of nodes to transitive models of $\ZFC$. Then $K(\M), v \Vdash x \in \mathrm{L}$ if and only if $\M_v \vDash x \in \mathrm{L}$, i.e., the formula $x \in \mathrm{L}$ is evaluated locally.
\end{proposition}
\begin{proof}
    Recall that the existential quantifier is defined locally, i.e., the witness for the quantification must be found within the domain associated to the current node in the Kripke model. Then, the statement of the proposition follows from the fact that $\Delta_0$-formulas are evaluated locally by \Cref{Proposition:Delta_0 locally}.
\end{proof}

The crucial detail of the following technical \Cref{Proposition:L frame} is the fact that the constructible universe is absolute between inner models of set theory. We will therefore need to strengthen the notion of a sound assignment. If $N$ and $M$ are transitive models of set theory, we say that $N$ is an \emph{inner model of $M$} if $N \subseteq M$, $N$ is a model of $\ZFC$, $N$ is a transitive class of $M$, and $N$ contains all the ordinals of $M$ (see \cite[p. 182]{Jech2003}).

\begin{definition}
    Let $K$ be a Kripke frame. We say that a sound assignment $\M: K \to V$ \emph{agrees on $\mathrm{L}$} if there is a transitive model $N \vDash \mathsf{ZFC} + V = \mathrm{L}$ such that $N$ is an inner model of $\M_v$ for every $v \in K$.
\end{definition}

In particular, if $K$ is a Kripke frame and $\M: K \to V$ agrees on $\mathrm{L}$, then we are justified in referring to the constructible universe $\mathrm{L}$ from the point of view of all models in $\M$.

\begin{lemma}\label{Proposition:L frame}
    Let $K$ be a Kripke frame and $\M$ be a sound assignment that agrees on $\mathrm{L}$. Then the following are equivalent for any formula $\phi(x)$ in the language of set theory, and all parameters $a_0,\dots,a_{n-1} \in \mathrm{L}$:
    \begin{enumerate}
        \item for all $v \in K$, we have $K(\M),v \Vdash (\phi(a_0,\dots,a_{n-1}))^\mathrm{L}$,
        \item for all $v \in K$, we have $\M_v \vDash (\phi(a_0,\dots,a_{n-1}))^\mathrm{L}$,
        \item there is a $v \in K$ such that $\M_v \vDash (\phi(a_0,\dots,a_{n-1}))^\mathrm{L}$, and,
        \item $\mathrm{L} \vDash \phi(a_0,\dots,a_{n-1})$.
    \end{enumerate}
\end{lemma}
\begin{proof} 
    By our assumption, $a_0,\dots,a_{n-1} \in \M_v$ for all $v \in K$  as $\mathrm{L} \subseteq \M_v$ for all $v \in K$. The equivalence of (ii), (iii) and (iv) follows directly from the fact that $\mathrm{L}$ is absolute between inner models of $\ZFC$. 
    
    The equivalence of (i) and (ii) can be proved by an induction on set-theoretic formulas simultaneously for all nodes in $K$ with the induction hypothesis as in the proof of \Cref{Proposition:Delta_0 locally}. For the case of the universal quantifier, we make use of the fact that $\M$ agrees on $\mathrm{L}$ (hence, that $\mathrm{L}$ is absolute between all models $\M_v$ for $v \in K$), and apply \Cref{Proposition: L locally}.
\end{proof}

\section{The Logical Structure of IKP}
\label{Section: The Logical Structure of IKP}

The aim of this section is to analyse the propositional and first-order logics of $\IKP$. First, we will introduce the logics of interest in a general way, and then proceed to introduce a Kripke model construction that we will use to determine certain logics of $\IKP^+$. 

\subsection{Logics and the De Jongh Property}

We will be concerned with both propositional and first-order logics.

\begin{definition}
    \label{Definition: Translation}
    A \emph{propositional translation $\sigma: \Prop \to \mathcal{L}_T^\sent$} is a map from propositional letters to sentences in the appropriate language that is extended to formulas in the obvious way.
    
    A \emph{first-order translation $\sigma: \mathcal{L}_J \to \mathcal{L}^\mathsf{form}_T$} is a map from the collection of relation symbols of $\mathcal{L}_J$ to $\mathcal{L}_\in$-formulas such that $n$-ary relation symbols are mapped to formulas with $n$-free variables. Then $\sigma$ is extended to all predicate formulas in $\mathcal{L}_J$ in the obvious way.
    
    If $J$ is a first-order logic with equality and $T$ a theory with equality, then a \emph{first-order equality translation $\sigma$} is a first-order translation with the extra condition that equality of $J$ is mapped to equality $T$.
\end{definition}

Following Visser \cite[Section 2.2]{Visser1999}, we only consider the case of predicate languages that contain only relation symbols by eliminating any function symbol $f$ by replacing it with a relation $R_f(x_0,\dots,x_n,y)$ defined by the equality $f(x_0,\dots,x_n) = y$. If we eliminate a function symbol in such a way, we demand that the interpreting theory $T$ proves that $\sigma(R_f)$ is the graph of a function (i.e., $\sigma$ being a translation is then dependent on the theory $T$). Nested function symbols can be eliminated with the usual procedure of introducing variables for the intermediate values.

Further, given a first-order logic $J$, we will make use of the relative translation $(\cdot)^E$ (where we shall always tacitly assume that $E$ is a fresh unary predicate symbol) that acts non-trivially only on quantifiers:
\begin{align*}
    (\exists x \ \phi(x))^E &= (\exists x (Ex \wedge \phi^E(x))), \text{ and,} \\
    (\forall x \ \phi(x))^E &= (\forall x (Ex \rightarrow \phi^E(x))).
\end{align*}

\begin{definition}
    Given a theory $T$, formulated in a language $\mathcal{L}_T$, we define the following logics:
    \begin{enumerate}
        \item The propositional logic $\Logic(T)$ of $T$ consists of the propositional formulas $\phi$ such that $T \vdash \phi^\sigma$ for all propositional translations $\sigma$.
        \item The first-order logic $\QLogic(T)$ of $T$ consists of the first-order formulas $\phi$ such that $T \vdash \phi^\sigma$ for all first-order translations $\sigma$ in the language $\mathcal{L}_T$.
        \item The relative first-order logic $\QLogic_E(T)$ of $T$ consists of the first-order formulas $\phi$ such that $T \vdash (\phi^E)^\sigma$ for all first-order translations $\sigma$ in the language $\mathcal{L}_T \cup \set{E}$, where $E$ is the fresh unary predicate symbol introduced for the relative translation.
        \item The first-order logic with equality $\QLogic^=(T)$ of $T$ consists of the first-order formulas $\phi$ such that $T \vdash \phi^\sigma$ for all first-order equality translations $\sigma$ in the language $\mathcal{L}_T$.
    \end{enumerate}
\end{definition}

\begin{definition}
    Let $T$ be a theory and $J$ a logic. We define the theory $T(J)$ as follows:
    \begin{enumerate}
        \item If $J$ is a propositional logic, we define $T(J)$ to be the theory obtained from $T$ by adding all sentences of the form $A^\sigma$ for formulas $A \in J$ and propositional translations $\sigma$.
        \item If $J$ is a first-order logic, we define $T(J)$ to be the theory obtained from $T$ by adding all sentences of the form $A^\sigma$ for formulas $A \in J$ and first-order translations $\sigma$.
    \end{enumerate}
\end{definition}

\begin{definition}
    \label{Definition: de Jongh property theorem}
    We say that a theory $T$ satisfies \emph{the de Jongh property for a logic $J$} if $\Logic(T(J)) = J$. A theory $T$ based on intuitionistic logic satisfies \emph{de Jongh's theorem} if $\Logic(T) = \IPC$.
\end{definition}

Before embarking on determining some logics of $\IKP^+$, let us survey a few known results. De Jongh \cite{DeJonghUnpublished,DeJongh1970} started the investigations of logics of arithmetical theories, establishing that $\Logic(\HA) = \IPC$ and $\QLogic_E(\HA) = \IQC$. De Jongh, Verbrugge and Visser \cite{deJongh2011} introduced the \emph{de Jongh property} and showed---among other results---that $\Logic(\HA(J)) = J$ for logics $J$ that are characterised by classes of finite frames. Considering a logic that is weaker than intuitionistic logic, Ardeshir and Mojtahedi \cite{ArdeshirMojtahedi2014} proved that the propositional logic of basic arithmetic is the basic propositional calculus.

For the sake of a counterexample to the de Jongh property, consider the theory $\mathsf{HA} + \mathsf{MP} + \mathsf{ECT}_0$, i.e., Heyting arithmetic extended with \emph{Markov's Principle} ($\mathsf{MP}$) and \emph{Extended Church's Thesis} ($\mathsf{ECT}_0$). Even though these principles are considered constructive, one can show that the propositional logic of this theory is an intermediate logic, i.e., $\IPC \subsetneq \Logic(\mathsf{HA} + \mathsf{MP} + \mathsf{ECT}_0) \subsetneq \CPC$ (this follows from results of Rose \cite{Rose1953} and McCarty \cite{McCarty1991}; for details see the discussion at the end of \cite[Section 2]{deJongh2011}). In conclusion, $\mathsf{HA} + \mathsf{MP} + \mathsf{ECT}_0$ does not satisfy de Jongh's theorem.

Turning now towards set theory, Passmann \cite{Passmann2020} used a Kripke-model construction to show that $\Logic(\IZF) = \Logic(\CZF) = \IPC$, and, in fact, that $\Logic(T(J)) = J$ for every set theory $T \subseteq \IZF$ and every logic $J$ characterised by a class of finite frames. H. Friedman and Ščedrov \cite{FriedmanScedrov1986} conclude from their earlier conservativity results \cite{FriedmanScedrov1985} that $\Logic(\mathsf{ZFI}) = \IPC$ holds for the two-sorted theory $\mathsf{ZFI}$. 

If $C$ is a class of formulas, we write $\Logic^C(T)$ for the propositional logic of $T$ where we restrict to the class of translations to maps $\sigma$ with $\ran(\sigma) \subseteq C$. We define $\QLogic^C(T)$ and $\QLogic_E^C(T)$ in the same way.

An important observation of H. Friedman and Ščedrov is the following.

\begin{theorem}[H. Friedman and Ščedrov, {\cite[Theorem 1.1]{FriedmanScedrov1986}}]
    \label{Theorem: Friedman Scedrov}
    Let $T$ be a set theory based on intuitionistic logic. Suppose that $T$ includes the axioms of Extensionality, Separation, Pairing and (finite) Union. Then $\IQC \subsetneq \QLogic(T)$, i.e., the first-order logic of $T$ is stronger than intuitionistic first-order logic.
\end{theorem}

This implies, in particular, that $\IZF$ does not satisfy the de Jongh property for $\IQC$. As $\CZF$ only contains $\Delta_0$-separation but full separation is used in the proof of the above theorem, the theorem does not apply to $\CZF$. However, with a slight adaption of the proof of H. Friedman and Ščedrov we can observe the following theorem. If $C$ is a class of formulas, we denote the separation scheme restricted to formulas from $C$ by \emph{$C$-Separation}.

\begin{theorem}
    Let $T$ be a set theory based on intuitionistic logic and $C$ be a class of formulas. Suppose that $T$ includes the axioms of Extensionality, $C$-Separation, Pairing and (finite) Union. Then $\IQC \subsetneq \QLogic^C(T)$, i.e., the $C$-first-order logic of $T$, $\QLogic^C(T)$, is stronger than intuitionistic first-order logic. 
\end{theorem}

So, in particular, $\IQC \subsetneq \QLogic^{\Delta_0}(\CZF)$, i.e., the $\Delta_0$-first-order logic of $\CZF$ is strictly stronger than intuitionistic logic. On the other hand, $A \vee \neg A \notin \QLogic^{\Delta_0}(\CZF)$, so $\IQC \subsetneq \QLogic^{\Delta_0}(\CZF) \subsetneq \CQC$.

\subsection{Constructing the Models}
\label{Subsection: Constructing the Models}

We will now introduce a class of Kripke models with classical domains that arise from certain classical models of set theory. These models will later be used to prove our results on logics of $\IKP$.

S. Friedman, Fuchino and Sakai \cite{FFS17} presented family of sentences that we are going to use to imitate the logical behaviour of a given Kripke frame. Consider the following statements $\psi_i$: 
    $$\text{There is an injection from } \aleph_{i+2}^\mathrm{L} \text{ to } \P(\aleph_i^\mathrm{L}).$$ 
There are different ways of formalising these statements that are classically equivalent, but (possibly) differ in the way they are evaluated in a Kripke model. For our purposes, we choose to define the sentence $\psi_i$ like this:
\begin{align*}
    \exists x \exists y \exists g ((x = \aleph_{i+2})^\mathrm{L} & \wedge (y = \aleph_i)^\mathrm{L} \\
    & \wedge g \text{ ``is an injective function'' } \\
    & \wedge \dom(g) = x \\
    & \wedge \forall \alpha \in x \forall z \in g(\alpha) \ z \in y )
\end{align*}
The main reason for this choice of formalisation is that the semantics of the existential quantifier is local, which will allow us to prove the following crucial observation. Note that each sentence $\psi_i$ is a $\Sigma_3$-formula.\footnote{It is clear that the final three conjuncts are $\Delta_0$-formulas. Using \Cref{Prop:L Sigma_1}, it is easy to check that the first two conjuncts are $\Pi_2$-formulas. In conclusion, the resulting formulas $\psi_i$ are $\Sigma_3$-formulas.}

\begin{proposition}\label{Proposition:Vdash iff vDash}
    Let $K$ be a Kripke frame and $\M$ a sound assignment that agrees on $\mathrm{L}$. Then $K(\M), v \Vdash \psi_i$ if and only if $\M_v \vDash \psi_i$, i.e., the sentences $\psi_i$ are evaluated locally.
\end{proposition}
\begin{proof}
    This follows from \Cref{Proposition:L frame}, \Cref{Proposition:Delta_0 locally} and the fact that the semantics of the existential quantifier is local, i.e., the sets $x$, $y$ and $g$ of the above statement must (or may not) be found within $\M_v$. In this situation, it suffices to argue that the following conjunction is evaluated locally:
    \begin{align*}
        (x = \aleph_{i+2})^\mathrm{L}  \wedge (y = \aleph_i)^\mathrm{L} & \wedge g \text{ ``is an injective function'' } \\
    & \wedge \dom(g) = x \\
    & \wedge \forall \alpha \in x \forall z \in g(\alpha) \ z \in y .
    \end{align*}
    It suffices to argue that every conjunct is evaluated locally. For the first two conjuncts of the form $\phi^\mathrm{L}$ this holds by \Cref{Proposition:L frame}. The final three conjuncts are $\Delta_0$-formulas. So we can apply \Cref{Proposition:Delta_0 locally} and the desired result follows.
\end{proof}

We will now obtain a collection of models of set theory using the forcing notions from S. Friedman, Fuchino and Sakai in \cite{FFS17}. From this collection, we define models with classical domains by constructing sound assignments that agree on $\mathrm{L}$. 

\begin{construction}
We begin by setting up the forcing construction. By our assumption that there is a countable transitive model of set theory, we can choose a minimal countable ordinal $\alpha$ such that $\mathrm{L}_\alpha$ is a model of $\ZFC + V = \mathrm{L}$. We fix this $\alpha$ for the rest of the article. Let $\mathbb{Q}_{\beta,n}$ be the forcing notion\footnote{The notation $\Fn(I,J,\lambda)$ is introduced by Kunen in \cite[Definition 6.1]{Kunen1980} and denotes the set of all partial functions $p:I \to J$ of cardinality less than $\lambda$ ordered by reversed inclusion.} $\Fn(\aleph_{\beta+n+2}^\mathrm{L},2,\aleph_{\beta + n}^\mathrm{L})$, defined within $\mathrm{L}_\alpha$. Given $A \subseteq \omega$, we define the following forcings:
 $$\mathbb{P}^A_{\beta,n} = \begin{cases}
    \mathbb{Q}_{\beta,n}, & \text{ if } n \in A, \\
    \mathbbm{1}, &\text{ otherwise.}
    \end{cases}$$
Then let $\mathbb{P}_\beta^A = \prod_{n < \omega} \mathbb{P}^A_{\beta,n}$ be the full support product of the forcing notions $\mathbb{P}^A_{\beta,n}$. Recall that the ordering $<$ on $\mathbb{P}_\beta^A$ is defined by $(a_i)_{i \in \omega} < (b_i)_{i \in \omega}$ if and only if $a_i <_i b_i$ for all $i \in \omega$. Now, let $G_\beta$ be $\mathbb{P}_\beta^\omega$-generic over $\mathrm{L}$, and let $G_{\beta,n} = \pi_n[G]$ be the $n$-th projection of $G_\beta$. Let $H$ be the trivial generic filter on the trivial forcing $\mathbbm{1}$. Now, for $A \subseteq \omega$ and $n \in \omega$ define the collection of filters:
$$
    G^A_{\beta,n} = \begin{cases}
        G_{\beta,n}, & \text{ if } n \in A, \\
        H, &\text{ otherwise,}
    \end{cases}
$$ and let $G_\beta^A = \prod_{n<\omega} G^A_{\beta,n}$.
\end{construction}

\begin{proposition}
    The filter $G_\beta^A$ is $\mathbb{P}_\beta^A$-generic over $\mathrm{L}_\alpha$. \qed
\end{proposition}

\begin{proposition}\label{Proposition:L[G^A] subseteq L[G^B]}
    If $A \subseteq B \subseteq \omega$ and $A \in \mathrm{L}[G_\beta^B]$, then $\mathrm{L}[G_\beta^A] \subseteq \mathrm{L}[G_\beta^B]$. Indeed, $\mathrm{L}[G_\beta^A]$ is an inner model of $\mathrm{L}[G_\beta^B]$. \qed
\end{proposition}

The additional assumption $A \in \mathrm{L}[G^B]$ is necessary because there are forcing extensions that cannot be amalgamated (see \cite[Observation 35]{FuchsHamkinsReitz} for a discussion of this). The following generalised proposition of S. Friedman, Fuchino and Sakai is crucial for our purposes.\footnote{In different terminology, the statement of the following proposition is that the sentences $\psi_i$ constitute a family of so-called \emph{independent buttons} for set-theoretical forcing. This terminology originates from the modal logic of forcing, see the article \cite{HamkinsLoewe2008} of Hamkins and Löwe.}

\begin{proposition}[S. Friedman, Fuchino and Sakai, {\cite[Proposition 5.1]{FFS17}}]\label{Proposition:LGA buttons}
    Let $\beta$ be an ordinal, $i \in \omega$ and $A \subseteq \omega$. Then $\mathrm{L}_\alpha[G_\beta^A] \vDash \psi_{\beta + i}$ if and only if $i \in A$. \qed
\end{proposition}
\begin{proof}
    S. Friedman, Fuchino and Sakai prove this proposition for the case $\beta = 0$. The generalised version can be proved in exactly the same way.
\end{proof}

This concludes our preparatory work, and we can state our main technical tool of this section as the following theorem.

\begin{theorem}
    \label{Theorem: Main Tool}
    Let $\beta < \alpha$ be an ordinal, $(K,\leq)$ be a Kripke frame and $f: K \to \P(\omega)$ be a monotone function such that $f(v) \in \mathrm{L}_\alpha$ for all $v \in K$. Then there is a sound assignment $\mathcal{M}$ that agrees on $\mathrm{L}_\alpha$ such that $K(\mathcal{M}), v \Vdash \psi_i$ if and only if there is $j \in f(v)$ such that $i = \beta + j$.
\end{theorem}
\begin{proof}
    Let $(K,\leq)$ be a Kripke frame and $f: K \to \P(\omega)$ be a function such that $f(v) \in \mathrm{L}_\alpha$ for all $v \in K$. Let $\mathcal{M}_v = \mathrm{L}_\alpha[G_\beta^{f(v)}]$. This is a well-defined sound assignment that agrees on $\mathrm{L}$ by \Cref{Proposition:L[G^A] subseteq L[G^B]}. By \Cref{Proposition:LGA buttons}, it holds that $i \in f(v)$ if and only if $\mathcal{M}_v \vDash \psi_i$. \Cref{Proposition:Vdash iff vDash} implies that the latter is equivalent to $K(\mathcal{M}),v \Vdash \psi_i$. The result follows.
\end{proof}

\subsection{Propositional Logics and IKP}
\label{Subsection: Propositional Logics and IKP}

We are now ready to prove a rather general result on the logics for which $\IKP$ satisfies the de Jongh property.\footnote{The results in this section are based on the third chapter of the second author's master's thesis \cite{Passmann2018}, supervised by Benedikt Löwe at the University of Amsterdam.} The essential idea is to transform a Kripke model for propositional logic into a Kripke model for set theory in such a way that the models exhibit very similar logical properties. In particular, if the logical model does not force a certain formula $\phi$, then we will construct a set-theoretic model and a translation $\tau$ such that the set-theoretic model will not force $\phi^\tau$.

Recall that an intermediate logic $J$ is called \emph{Kripke-complete} if there is a class of Kripke frames $C$ such that $J = \Logic(C)$.

\begin{theorem}
    Let $T \subseteq \IKP^+ + \MP + \AC$ be a set theory. If $J$ is a Kripke-complete intermediate propositional logic, then $\Logic^{\Sigma_3}(T(J)) = J$.
\end{theorem}
\begin{proof}
    The inclusion from right to left follows directly from the definition of $T(J)$. We show the converse inclusion by contraposition. So assume that there is a formula $\phi$ in the language of propositional logic such that $J \not \vdash \phi$. By our assumption that $J$ is Kripke-complete, there is a Kripke model $(K,\leq,V)$ such that $(K,\leq,V) \Vdash J$ but $(K,\leq,V) \not \Vdash \phi$. Without loss of generality, we can assume that the propositional letters appearing in $\phi$ are $p_0,\dots,p_n$. We define a function $f: K \to \P(\mathbb{N})$ by stipulating that:
    $$
        i \in f(v) \text{ if and only if } i \leq n \text{ and } (K,\leq,V), v \Vdash p_i.
    $$
    In particular, $f(v)$ is finite and thus $f(v) \in \mathrm{L}$ for every $v \in K$. Apply \Cref{Theorem: Main Tool} to get a sound assignment $\mathcal{M}$ that agrees on $\mathcal{L}$ such that $K(\mathcal{M}), v \Vdash \psi_i$ if and only if $i \in f(v)$. 

    Let $\sigma: \Prop \to \mathcal{L}_\in^\sent$ be the map $p_i \mapsto \psi_i$. It follows via an easy induction on propositional formulas that $K(\mathcal{M}), v \Vdash \chi^\sigma$ if and only if $(K,\leq,V),v \Vdash \chi$. In particular, $K(\mathcal{M}) \not \Vdash \phi^\sigma$ but $K(\mathcal{M}) \Vdash T(J)$. Hence, $\phi \notin \Logic(T(J))$.
\end{proof}

\begin{corollary}
    \label{Corollary: Propositional Logic IKP}
    Every set theory $T \subseteq \IKP^+ + \MP + \AC$ has the de Jongh property with respect to every Kripke-complete intermediate propositional logic $J$, i.e., $\Logic(T(J)) = J$. \qed
\end{corollary}

De Jongh, Verbrugge and Visser \cite{deJongh2011} proved a similar result for Heyting arithmetic $\HA$, namely, that $\Logic(\HA(J)) = J$ holds for every intermediate propositional logic $J$ which possesses the finite frame property. Passmann \cite{Passmann2020} showed that $\Logic(\IZF(J)) = J$ holds for every intermediate logic $J$ that is complete with respect to a class of finite trees. Our present \Cref{Corollary: Propositional Logic IKP}, however, applies to a much broader class of logics: all intermediate logics that are complete with respect to a class of Kripke frames.

\subsection{The Relative First-Order Logic of IKP}

When it comes to first-order logics, several intricacies arise that concern the interplay of the logics and the surrounding set theory. We were able to ignore these intricacies in the previous section when we were dealing with propositional logics because we effectively reduced the problem to finitely many propositional letters. In the case of first-order logic, however, we need to deal with infinite domains and predication. 

The basic idea remains the same: We will construct a set-theoretical model based on a Kripke model for first-order logic. This time, however, we also need to deal with domains and predication. We will see that working with relative interpretations allows us to easily adapt the proof of the previous section for our purposes here: We will use the statements $\psi_i$ to code domains of Kripke models for $\IQC$ as subsets of $\omega$ as well as coding which predications hold true.

\begin{theorem}
    \label{Theorem: L_alpha first-order Existence de Jongh property}
    Let $T \subseteq \IKP^+ + \MP + \AC$ be a set theory. If $J \in \mathrm{L}_\alpha$ is an intermediate first-order logic such that $\mathrm{L}_\alpha \vDash$ ``$J$ is a Kripke-complete logic in a countable language'', then $\QLogic^{\Sigma_3}_E(T(J)) = J$.
\end{theorem}
\begin{proof}
    Again, the inclusion from right to left is trivial and we prove the other direction by contraposition. 
    
    Let $J \in \mathrm{L}_\alpha$ be a first-order logic such that ``$J$ is Kripke-complete'' holds in $\mathrm{L}_\alpha$. Let $J \not \vdash \phi$ for some first-order sentence $\phi$. We have to find a map $\sigma$ such that $T(J) \not \vdash (\phi^E)^\sigma$.
    
    Work in $\mathrm{L}_\alpha$. By the fact that $J \not \vdash \phi$ and that $J$ is Kripke-complete, we know that there is first-order Kripke model $M = (K,\leq,D,I) \in L_\alpha$ such that $M \not \Vdash \phi$. As we work in a classical meta-theory, we apply the downward Löwenheim-Skolem-Theorem by coding $M$ as first-order structure and assume without loss of generality that $M$ is countable. Fix enumerations $d: \omega \to \bigcup{D}$ of the union of all domains of the model $M$, $C: \omega \to \mathcal{L}_{J}$ of all constant symbols appearing in $\mathcal{L}_J$, $R: \omega \to \mathcal{L}_{J}$ of all relation symbols appearing in $\mathcal{L}_J$, and $F: \omega \to \mathcal{L}_{J}$ of all function symbols appearing in $\mathcal{L}_J$ (in each case, if there are only finitely many symbols, restrict the domain to some $n \in \omega$).
    
    Still working in $\mathrm{L}_\alpha$, we will now code all information about $M$ in sets of natural numbers. Without loss of generality, we can assume that $D_v \subseteq \omega$ for all $v \in K$, and that the transition functions are inclusions. Fix now a map $\seq{\cdot}: \omega^{<\omega} \to \omega$. For $v \in K$, we let $k \in f(v)$ if and only if one of the following cases holds true:
    \begin{enumerate}
        \item $k = \seq{0,j}$ and $j \in D_v$,
        \item $k = \seq{1,i,j_0,\dots,j_{n-1}}$, $R_i$ is an $n$-ary relation symbol, $j_0,\dots,j_{n-1} \in D_v$ and 
            $$v \Vdash R_i(j_0,\dots,j_{n-1}).$$
    \end{enumerate}
	Observe that we have defined a function $f: K \to \P(\omega)$. This $f$ is monotone due to the persistence property of Kripke models. 
    
    Now work in $V$, and apply \Cref{Theorem: Main Tool} to obtain a sounds assignment $\mathcal{M}$ that agrees on $\mathrm{L}$ such that $K(\mathcal{M}),v \Vdash \psi_k$ if and only if $k \in f(v)$. We define a translation $\sigma$:
    \begin{enumerate}
        \item if $\chi = Et$, where $E$ is the predicate of the relative translation, then $(Et)^\sigma = \psi_{\seq{0,t^\sigma}}$, and,
        \item if $R_i(t_0,\dots,t_{n-1})$ is an $n$-ary relation symbol different from the existential predicate $E$, then $R_i(t_0,\dots,t_{n-1})^\sigma = \psi_{\seq{1,i,t_0^\sigma,\dots,t_{n-1}^\sigma}}$.
    \end{enumerate}
    
    Note that the sentences $\psi_i$ are uniformly defined for $i \in \omega$, and therefore the translation $\sigma$ is well-defined. With an easy induction on formulas $\chi$ in the language of $J$ we show that $K(\mathcal{M}), v \Vdash \chi^\sigma$ if and only if $M,v \Vdash \chi$. We can then conclude that $K(\M) \not \Vdash \phi^\sigma$ but $K(\M) \Vdash J^\sigma$, i.e., $\phi \notin \QLogic_E(T)$.
\end{proof}

The following corollary shows that the theorem covers many important cases. Recall that a logic is \emph{axiomatisable} if it has a recursively enumberable axiomatisation.

\begin{corollary}
    \label{Corollary: axiomatisable logics relative case}
	Let $T \subseteq \IKP^+ + \MP + \AC$ be a set theory.
    If $J$ is an axiomatisable intermediate first-order logic that is $\ZFC$-provably Kripke-complete, then $\QLogic_E^{\Sigma_3}(T(J)) = J$.
\end{corollary}
\begin{proof}
    By Craig's Lemma \cite{Craig1953}, we know that axiomatisable logics are recursively axiomatisable, i.e., we can assume without loss of generality that $J$ is a recursive set. As recursive sets are $\Delta^0_1$-definable with parameter $\omega$ (as a coding of Turing machines in arithmetic), it follows that $J \in \mathrm{L}_{\omega+2} \subseteq \mathrm{L}_\alpha$. Hence, we can apply \Cref{Theorem: L_alpha first-order Existence de Jongh property} and derive the desired result.
\end{proof}

\begin{corollary}
    \label{Corollary: Relative First-Order Logic IKP}
    Let $T \subseteq \IKP^+ + \MP + \AC$ be a set theory. The relative first-order logic of $T$ is intuitionistic first-order logic $\IQC$, i.e., $\QLogic_E(T) = \IQC$. In particular, $\QLogic_E(\IKP) = \IQC$. \qed
\end{corollary}

We give a few more examples of logics to which \Cref{Corollary: axiomatisable logics relative case} applies. To this end, note that $\mathrm{KF}$ is the following scheme:
$$
    \neg \neg \forall x (P(x) \vee \neg P(x)).
$$
Moreover, $\mathbf{QHP}_k$ is the first-order logic of frames of depth at most $k$, and $\mathbf{QLC}$ is the first-order logic of linear frames. For more on these logics, we refer the reader to the book of Gabbay, Shehtman and Skvortsov \cite{GabbayShehtmanSkvortsov2009}.

\begin{corollary}
    \label{Corollary: QLogic_E more logics}
    Let $T \subseteq \IKP^+ + \MP + \AC$ be a set theory. It holds that $\QLogic_E^{\Sigma_3}(T(J)) = J$ in case that $J$ is one of $\mathbf{IQC}+\mathrm{KF}$, $\mathbf{QHP}_k$ or $\mathbf{QLC}$.
\end{corollary}
\begin{proof}
    This follows from \Cref{Corollary: axiomatisable logics relative case} and the respective completeness theorems from \cite{GabbayShehtmanSkvortsov2009} (see \cite[Theorem 6.3.5]{GabbayShehtmanSkvortsov2009} for the completeness of $\mathbf{IQC}+\mathrm{KF}$, \cite[Theorem 6.3.8]{GabbayShehtmanSkvortsov2009} for completeness of $\mathbf{QHP}_k$, and \cite[Theorem 6.7.1]{GabbayShehtmanSkvortsov2009} for completeness of $\mathbf{QLC}$).
\end{proof}

\subsection{The First-Order Logic of IKP}

The most important result of this section is that the first-order logic of $\IKP$ is intuitionistic first-order logic, i.e., $\QLogic(\IKP) = \IQC$. We will show this by generalising the argument of the previous sections. Our first step will be to construct the necessary Kripke models for set theory. 

Our approach in this section will be somewhat different from what we did in the previous two sections. As we now have to deal with unrestricted quantification, we have to give up on the idea of coding directly into the classical models $\mathcal{M}_v$ which propositions or predications must be true at a certain node. Rather, the idea is we will now encode enough information such that the models know internally which predication must hold at which node. We remind the reader that we consider $\IQC$ to be intuitionistic first-order logic \emph{without} equality.

\begin{construction} 
    Recall that we take $\mathrm{L}_\alpha$ to be the least transitive model of $\ZFC + V = L$. Let $(K,\leq,D,V) \in \mathrm{L}_\alpha$ be a well-founded rooted Kripke model for $\IQC$. Work in $\mathrm{L}_\alpha$. By \Cref{Lemma: Kripke Model Extension of Domains} we can assume that there is a rooted well-founded countable Kripke model $(K,\leq,D,V)$ with countably increasing domains. Without loss of generality, we may assume that $D_v \subseteq \omega$ for all $v \in K$, and $D_v \subseteq D_w$ for $v \leq w$. Let $D_v^* = D_v \setminus \bigcup_{w < v} D_w$. We can assume by well-foundedness and shuffling of the domains, if necessary, that for every $x \in \bigcup_{v \in K} D_v$, there is a unique node $v_x \in K$ with $x \in D_{v_x}^*$. As $K$ is countable, we can take an injective function $f: K \to \omega \setminus \set{0}$. Moreover, for every node $v \in K$ let $f_v: \omega \to D_v^*$ be the unique order preserving enumeration of $D_v^*$. Define a function $t: K \times \omega \to \mathcal{P}((\omega \times \omega)^{<\omega})$ such that for every $n$-ary predicate $P$ and $v \in K$:
$$
    t(v,\godel{P}) = \Set{((f_{v_{x_0}}^{-1}(x_0),v_{x_0}),\dots,(f_{v_{x_n}}^{-1}(x_n),v_{x_n}))}{v \Vdash P(x_0,\dots,x_n)}
$$

By $V = L$, there is an ordinal $\gamma$ such that the tuple $(K,\leq,f,t)$ is the $\gamma$-th element in the canonical well-ordering of $L$. Let $F: K \to \P(\omega)$ be the function such that $F(v) = \Set{f(w)}{w \leq v} \cup \set{0}$. 

By \Cref{Theorem: Main Tool}, there is a Kripke model $K(\mathcal{M})$ with classical domains such that $K(\mathcal{M}), v \Vdash \psi_i$ if and only if there is $j \in F(v)$ such that $i = \gamma + j$.
\end{construction}

\begin{definition}
    We call $K(\mathcal{M})$ a \emph{mimic model} of $M = (K,\leq,D,V)$. Further, we say that $\gamma$ is the \emph{essential ordinal} of the mimic model $K(\mathcal{M})$.
\end{definition}

We will sometimes refer to $(K,\leq,f,t)$ as the \emph{coded model} of $K(\mathcal{M})$. In the following series of lemmas, we will spell out the way in which the mimic models can recover the information about the coded model.

\begin{lemma}
    There is a $\Sigma_3$-formula $\phi_\mathsf{ess}(x)$ in the language of set theory such that $K(\mathcal{M}),v \Vdash \phi_\mathsf{ess}(x)$ if and only if $x$ is the essential ordinal $\gamma$ of $K(\mathcal{M})$.
\end{lemma}
\begin{proof}
    We define the formula $\phi_\mathsf{ess}(x)$ as follows:
    \begin{equation*}
        \phi_\mathsf{ess}(x) \equiv x \in \Ord \wedge \psi_x \wedge \forall \beta \in x \neg \psi_\beta
    \end{equation*}
    By the definition of the mimic model $K(\mathcal{M})$, we know that $K(\mathcal{M}) \Vdash \psi_\gamma$ and $K(\mathcal{M}) \not \Vdash \psi_i$ for $i < \gamma$, i.e., $K(\mathcal{M}) \Vdash \neg \psi_i$ for all $i < \gamma$. As being an ordinal can be expressed by a $\Delta_0$-formula, it follows that $K(\mathcal{M}) \Vdash \gamma \in \Ord \wedge \psi_\gamma \wedge \forall \beta \in \gamma \ \neg \psi_\beta$. 
    
    Conversely, if $K(\mathcal{M}), v \Vdash \phi_\mathsf{ess}(x)$, then it follows that $x \in \mathcal{M}_v$ is an ordinal such that $\mathcal{M}_v \vDash \psi_x$ and for all $\beta < x$ and $w \geq v$ we have $\mathcal{M}_v \vDash \neg \psi_\beta$. By the definition of $K(\mathcal{M})$ it must hold that $x = \gamma$.
\end{proof}

\begin{lemma}
    Let $\gamma$ be the essential ordinal. There is a $\Sigma_1$-formula $\phi_\mathsf{orig}(x,y)$ in the language of set theory such that $K(\mathcal{M}),v \Vdash \phi_\mathsf{orig}(x,\gamma)$ if and only if $x$ is the coded model of $K(\mathcal{M})$ (i.e., $x = (K,\leq,f,t)$).
\end{lemma}
\begin{proof}
    Consider the following formula:
    \begin{equation*}
    \phi_\mathsf{orig}(x,y) \equiv \text{``$x$ is the $y$-th element in the canonical well-ordering of $\mathrm{L}$''}^\mathrm{L}    
    \end{equation*}
    Now, by \Cref{Proposition:L frame}, $K(\mathcal{M}), v \Vdash \phi_\mathsf{orig}(x,\gamma)$ is equivalent to
    $$
        \mathrm{L}_\alpha \vDash \text{``$x$ is the $\gamma$-th element int he canonical well-rodering of $\mathrm{L}$''}.
    $$
    The definition of the essential ordinal ensures that this is the case if and only if $x = (K,\leq,f,t)$. To observe that $\phi_\mathsf{orig}(x,y)$ is a $\Sigma_1$-formula use the fact that the canonical well-ordering of $\mathrm{L}$ is $\Sigma_1$-definable (see \cite[Lemma 13.19]{Jech2003}).
\end{proof}

\begin{lemma}
    There is a $\Sigma_3$-formula $\phi_\mathsf{exists}(x,y)$ in the language of set theory, using the coded model $(K,\leq,f,t)$ and the essential ordinal $\gamma$ as parameters, such that $K(\mathcal{M}),v \Vdash \phi_\mathsf{exists}(x,y)$ if and only if $y \in K$ such that $y \leq v$ and $x \in \mathcal{M}_y$.
\end{lemma}
\begin{proof}
    Recall from \Cref{Subsection: Constructing the Models} that $\M_v$ is the model $\mathrm{L}_\alpha[G_\gamma^{A_v}]$ where $G_\gamma^{A_v}$ is $\mathrm{L}_\alpha$-generic for $\mathbb{P}^{A_v}_\gamma$ and $A_v = \set{0} \cup \Set{f(w)}{w \leq v}$. Consider the following formula:
    \begin{align*}
        \phi_\mathsf{exists}(x,y) \equiv \ & \exists \mathbb{P} \in \mathrm{L} (\text{``$\mathbb{P} = \mathbb{P}^A_\gamma$ where $A = \set{0} \cup \Set{f(w)}{w \in K \wedge w \leq y}$''}^\mathrm{L} \\
               & \wedge \exists \tau \in \mathrm{L} (\text{``$\tau$ is a $\mathbb{P}$-name''} \wedge \exists G (G \text{ is generic for } \mathbb{P} \text{ and } \tau^G = x)) ).
    \end{align*}
    Note the use of the parameters $(K,\leq,f,t)$ and $\gamma$, and observe that this formula is evaluated locally as it is constructed from $\Delta_0$-formulas, formulas relativised to $\mathrm{L}$ and existential quantification. 
    
    Let $w \in K$ such that $w \leq v$. By general facts about set-theoretical forcing, $x \in \mathcal{M}_w = \mathsf{L}_\alpha[G^{A_w}_\gamma]$ if and only if there exists a $\mathbb{P}^{A_w}_\gamma$-name $\tau \in \mathsf{L}_\alpha$ such that $\tau^{G^{A_w}_\gamma} = x$. Equivalently, $\mathcal{M}_v \vDash \phi_\mathsf{exists}(x,w)$, and in turn holds if and only if $K(\mathcal{M}), v \Vdash \phi_\mathsf{exists}(x,w)$, by our observation on local evaluation. 
\end{proof}

For the next lemma, we introduce some handy notation. Let $\mathcal{M}_v^* = \mathcal{M}_v \setminus \bigcup_{w < v} \mathcal{M}_w$.

\begin{lemma}
    There is a $\Sigma_3$-formula $\phi_\mathsf{birth}(x,y)$ in the language of set theory, using the coded model $(K,\leq,f,t)$ and the essential ordinal $\gamma$ as parameters, such that $K(\mathcal{M}),v \Vdash \phi_\mathsf{birth}(x,y)$ if and only if $y \in K$ such that $y \leq v$ and $x \in \mathcal{M}_y^*$.
\end{lemma}
\begin{proof}
    Let $\phi_\mathsf{birth}(x,y)$ be defined as follows:
    \begin{equation*}
        \phi_\mathsf{birth}(x,y) \equiv y \in K \wedge \phi_\mathsf{exists}(x,y) \wedge \forall u \in K (u < y \rightarrow \neg \phi_\mathsf{exists}(x,u)).
    \end{equation*}
    If $w \leq v$ and $x \in \mathcal{M}_w^*$, then it follows from the previous lemma that for all $u < w$, $K(\mathcal{M}) \not \Vdash \phi_\mathsf{exists}(x,u)$, i.e., $K(\mathcal{M}) \Vdash \neg \phi_\mathsf{exists}(x,u)$. On the other hand, we clearly have $v \Vdash \phi_\mathsf{exists}(x,w)$ and hence $v \Vdash \phi_\mathsf{birth}(x,w)$.
    
    Conversely, if $v \Vdash \phi_\mathsf{birth}(x,w)$ for $w \leq v$, it follows that $x \in \mathcal{M}_w$ but $x \notin \mathcal{M}_u$ for $u < w$. Hence, $x \in \mathcal{M}_w^*$.
\end{proof}

\begin{lemma}
    There is a $\Sigma_3$-formula $\phi_\mathsf{passed}(x)$ in the language of set theory, using the coded model $(K,\leq,f,t)$ as a parameter, such that $K(\mathcal{M}), v \Vdash \phi_\mathsf{passed}(x)$ if and only if $x \in K$ such that $x \leq v$.
\end{lemma}
\begin{proof}
    Consider the following formula:
    \begin{equation*}
        \phi_\mathsf{passed}(x) \equiv \psi_{f(x)}
    \end{equation*}
    The lemma now follows directly from the definition of the mimic model $K(\mathcal{M})$.
\end{proof}

We have now finished our preparations and can prove the following lemma which will show that the mimic model can imitate the predication of the coded model. This is a crucial step for connecting truth in the mimic model with truth in the coded model. 

Given $x \in \mathcal{M}_v$, let $v_x \in K$ be the unique node with $x \in D_{v_x}^*$ and $r_x \in \omega$ such that $\rank(x) = \lambda + r_x$ for some limit ordinal $\lambda$. Define a map $g_v: \mathcal{M}_v \to D_v$ by $g_v(x) = f_{v_x}(r_x)$. Further let $\godel{\cdot}: \mathcal{L}_\IQC \to \omega$ be a fixed Gödel coding function.

\begin{lemma}
    \label{Lemma: Predicate translation}
    Let $P$ be an $n$-ary predicate. There is a $\Sigma_3$-formula $\phi_P(\bar x)$ with parameters only $x_0,\dots,x_{n-1}$ in the language of set theory such that $K(\mathcal{M}), v \Vdash \phi_P(x_0,\dots,x_{n-1})$ if and only if $(K,\leq,D,V), v \Vdash P(g_v(x_0),\dots,g_v(x_{n-1}))$.
\end{lemma}
\begin{proof}
    Let $\phi_P(x_0,\dots,x_n)$ be the following formula:
    \begin{align*}
        \exists K, \leq, f, t, \bar r, \bar u, w, \gamma ( &\phi_\mathsf{ess}(\gamma) \wedge \phi_\mathsf{orig}(K,\leq,f,t) \\
        &\wedge \bigwedge_{i < n}(\exists \lambda (\text{``$\lambda$ limit ordinal''} \wedge r_i \in \omega \wedge \rank(x_i) = \lambda + r_i)) \\
        &\wedge \bigwedge_{i < n} \phi_\mathsf{birth}(x_i,u_i) \wedge \phi_\mathsf{passed}(w) \\
        &\wedge \bigwedge_{i < n} w \geq u_i \\
        &\wedge ((r_0,u_0), \dots, (r_{n-1},u_{n-1})) \in t(w,\godel{P})
        ).
    \end{align*}
    Unfolding the formula by using the sequence of lemmas proved above, we see that $K(\mathcal{M}),v \Vdash \phi_P(x_0,\dots,x_n)$ is equivalent to the existence of some $w \leq v$ such that there are $u_i \leq w$ with $x_i \in \mathcal{M}_{u_i}^*$, $r_i \in \omega$ such that $\rank(x_i) = \lambda_i + r_i$ for some limit ordinals $\lambda$ and $((r_0,u_0), \dots, (r_{n-1},u_{n-1})) \in t(w, \godel{P})$. By definition of $t$, this is equivalent to $(K,\leq,D,V), w \Vdash P(f_{u_0}(x_0),\dots,f_{u_{n-1}}(x_{n-1}))$, and hence $(K,\leq,D,V), w \Vdash P(g_v(x_0),\dots,g_v(x_{n-1}))$ by definition of $g_v$. Persistency implies $(K,\leq,D,V), v \Vdash P(g_v(x_0),\dots,g_v(x_{n-1}))$.
    
    Conversely, if $(K,\leq,D,V), v \Vdash P(g_v(x_0), \dots, g_v(x_n))$, then by definition of $g_v$, $(K,\leq,D,V), v \Vdash P(f_{v_{x_0}}(r_0), \dots, f_{v_{x_n}}(r_n))$, where $r_i \in \omega$ are as above. By definition of $t$, we will have that:
    \begin{align*}
        & ((r_0,v_{x_0}), \dots, (r_{n-1},v_{x_{n-1}})) \\
        = & ((f_{v_{x_0}}^{-1}(f_{v_{x_0}}(r_0)),v_{x_0}), \dots, (f_{v_{x_n}}^{-1}(f_{v_{x_{n-1}}}(r_{n-1}),v_{x_{n-1}})) \in t(v,\godel{P})
    \end{align*}
    It follows that $K(\mathcal{M}), v \Vdash \phi_P(x_0,\dots,x_{n-1})$.
\end{proof}

Define a first-order translation $\tau: \mathcal{L}_\IQC \to \mathcal{L}_\in$ by stipulating that $P(\bar x)^\tau = \phi_P(\bar x)$. Note that the range of $\tau$ consists of $\Sigma_3$-formulas. We can now extend the correspondence to all first-order formulas.

\begin{lemma}
    \label{Lemma: Mimic model translation}
    Let $K(\mathcal{M})$ be a mimic model of a well-founded rooted Kripke model $(K,\leq,D,V)$ for $\IQC$. For every formula $\phi$ in the language of first-order logic, we have that $K(\mathcal{M}), v \Vdash \phi(\bar x)^\tau$ if and only if $(K,\leq,D,V), v \Vdash \phi(g_v(\bar x))$.
\end{lemma}
\begin{proof}
    This is proved by an induction on the complexity of $\phi$ for all $v \in K$. The atomic cases has been taken care of in \Cref{Lemma: Predicate translation} and the cases for the logical connectives $\vee$, $\wedge$ and $\rightarrow$ follow trivially. We will now prove the cases for the quantifiers. 
    
    First observe that the maps defined by $g_v$ are surjective. This is due to the fact that $\mathcal{M}_v^*$ contains elements of rank $\lambda + n$ for any $n < \omega$.\footnote{This can be shown via a construction starting with the generic $x_0 := G \in \mathcal{M}_v^*$ and iterating the operation $x_{n_1} := \set{x_n}$. Then take $y_0 = \bigcup_{n<\omega} x_n$ and $y_{n+1} := \set{y_n}$. It follows that $y_n$ has rank $\lambda + n$ for some limit ordinal $\lambda$.}  
    
    For the existential quantifier, assume that $K(\mathcal{M}), v \Vdash (\exists x \phi(x, \bar z))^\tau$. This is equivalent to the existence of some $x \in \mathcal{M}_v$ such that $K(\mathcal{M}), v \Vdash \phi^\tau(x, \bar z)$. By induction hypothesis, this is equivalent to the existence of some $x \in \mathcal{M}_v$ such that $(K,\leq,D,V), v \Vdash \phi(g_v(x), g_v(\bar z)))$. By the fact that $g_v$ is surjective, we know that the latter is equivalent to $(K,\leq,D,V), v \Vdash \exists x \phi(x, g_v(\bar z))$.
    
    For the universal quantifier, observe that $K(\mathcal{M}), v \Vdash (\forall x \phi(x, \bar z)))^\tau$ is equivalent to the fact that for all $x \in \mathcal{M}_v$ it holds that $K(\mathcal{M}), v \Vdash \phi^\tau(x, g_v( \bar z))$. By induction hypothesis this holds if and only if for all $x \in \mathcal{M}_v$ we have $(K,\leq,D,V), v \Vdash \phi(g_v(x), g_v(\bar z))$. Again, by using the surjectivity of $g_v$, this is equivalent to $(K,\leq,D,V), v \Vdash \forall x \phi(x, g_v(\bar z))$.
\end{proof}

We are now ready to derive the final result.

\begin{theorem}
    \label{Theorem: QLogic(IKP^+(J)) = J}
    Let $T \subseteq \IKP^+ + \MP + \AC$ be a set theory. If $J \in \mathrm{L}_\alpha$ is an intermediate first-order logic that is $\ZFC$-provably Kripke-complete with respect to a class of well-founded frames, then $\QLogic^{\Sigma_3}(T(J)) = J$.
\end{theorem}
\begin{proof}
    Let $J \in \mathrm{L}_\alpha$ be $\ZFC$-provably Kripke-complete first-order logic. It is clear that $J \subseteq \QLogic(T(J))$. For the other direction, assume that $J \not \vdash \phi$. By our assumptions, there is a Kripke model $(K,\leq,D,V) \in \mathrm{L}_\alpha$ such that $(K,\leq,D,V) \not \Vdash \phi$. Due to \Cref{Lemma: Kripke Model Extension of Domains} we can assume without loss of generality that $(K,\leq,D,V)$ has countably increasing domains. Let $K(\mathcal{M})$ be a mimic model obtained from $(K,\leq,D,V)$. By \Cref{Lemma: Mimic model translation} it follows that $K(\mathcal{M}), v \not \Vdash \phi^\tau$. As $K(\mathcal{M})$ is a model of $\IKP^+$ and $T \subseteq \IKP^+$, it follows that $\IKP^+ \not \vdash \phi^\tau$ so that $\phi \notin \QLogic(\IKP^+)$. This finishes the proof of the theorem.
\end{proof}

We conclude this section by stating some important corollaries.

\begin{corollary}
    Let $T \subseteq \IKP^+ + \MP + \AC$ be a set theory. If $J \in \mathrm{L}_\alpha$ is an intermediate first-order logic that is $\ZFC$-provably Kripke-complete with respect to a class of well-founded frames, then $\QLogic(T(J)) = J$.
\end{corollary}
\begin{proof}
    As in the proof of \Cref{Corollary: axiomatisable logics relative case}, we use the fact that every axiomatisable first-order logic is contained in $\mathrm{L}_\alpha$. The result then follows with \Cref{Theorem: QLogic(IKP^+(J)) = J}.
\end{proof}

\begin{corollary}
    \label{Corollary: First-Order Logic IKP}
    Let $T \subseteq \IKP^+ + \MP + \AC$ be a set theory. The first-order logic of $T$ is $\IQC$, i.e., $\QLogic(T) = \IQC$. In particular, $\QLogic(\IKP) = \IQC$.
\end{corollary}
\begin{proof}
     This follows from the fact that $\IQC$ is $\ZFC$-provably Kripke-complete with respect to a class of well-founded frames (see the proof of \cite[Theorem 8.17]{Takeuti1987}), and applying the previous corollary.
\end{proof}

\begin{corollary}
    Let $T \subseteq \IKP^+ + \MP + \AC$ be a set theory. Then $\QLogic(T(\mathbf{QHP}_k)) = \mathbf{QHP}_k$ for $k < \omega$.
\end{corollary}
\begin{proof}
    This follows from the fact that $\mathbf{QHP}_k$ is complete with respect to the class of frames of depth at most $k$ (see \cite[Theorem 6.3.8]{GabbayShehtmanSkvortsov2009}).
\end{proof}

\subsection{The First-Order Logic with Equality of IKP}

Given the results in the previous section, a natural question would be whether these results extend to logic \emph{with} equality. In this section we show that that is not the case. 

\begin{theorem}
    Let $T$ be a set theory based on intuitionistic logic containing the axioms of extensionality, empty set and pairing. Then the first-order logic with equality of $T$, $\QLogic^=(T)$, is strictly stronger than $\IQC^=$, i.e., $\IQC^= \subsetneq \QLogic^=(T)$.
\end{theorem}
\begin{proof}
    Let $\phi$ denote the following formula in the language of $\IQC^=$:
    $$
        [\exists x \exists y \forall z (z = x \vee z = y)] \rightarrow [\exists x \forall z (z = x)].
    $$
    Intuitively, $\phi$ formalises the statement ``if there are at most two objects, then there is at most one object.'' Note that $\phi^\sigma = \phi$ holds for any first-order equality translation $\sigma$ into the language of set theory. By the principle of \emph{ex falso quodlibet} it therefore suffices to show that the antecedent of $\phi$ is false in $T$. Let us call this antecedent $\psi.$
    
    We give an informal argument that can be easily transferred into a formal proof in the theory $T$. By pairing and emptyset, we can obtain the sets $0 = \emptyset$, $1 = \set{\emptyset}$, and $2 = \set{\emptyset,\set{\emptyset}}$. Suppose $\psi$. Then, by transitivity of equality, we know that $0 = 1 \vee 0 = 2 \vee 1 = 2$ must hold. In each case, we can derive falsum, $\bot$, using extensionality and the empty set axiom. With $\vee$-elimination and $\rightarrow$-introduction, we conclude that $\neg \psi$ holds.
    
    This argument shows that $\phi \in \QLogic^=(T)$. To finish the proof of the theorem, it is enough to show that $\phi \notin \IQC^=$. This follows by completeness as follows. Consider the Kripke model for $\IQC^=$ that consists of one node with a domain of two distinct points: the antecedent of $\psi$ will be true in this model but the consequent fails.
\end{proof}

\begin{corollary}
    \label{Corollary: First-order logic with equality of IKP is stronger}
    The first-order logic with equality of any set theory $T$ considered in this paper, such as $\IKP$, $\IKP^+ + \MP + \AC$, $\CZF$ and $\IZF$, is stronger than $\IQC^=$. \qed
\end{corollary}

We close this section with the following question.

\begin{question}
    What is the first-order logic with equality $\QLogic^=(\IKP)$ of $\IKP$?
\end{question}

\section{Conclusions and Open Questions}
We have seen that $\IKP$ is a very well-behaved theory from the logical point of view---in fact, this applies to every subtheory of $\IKP^+ + \MP + \AC$. This result is also conceptually important: Constructive set theories are usually formulated on the basis of $\IQC$ in such a way that the set-theoretic axioms should not strengthen the logic, ensuring \emph{intuitionistic reasoning}. We have shown that $\IKP$ indeed satisfies this requirement. Determining the first-order logic of $\CZF$, and thus ensuring that $\CZF$ is logically and conceptually well-behaved, is an open problem.

\begin{question}
    What is the first-order logic of $\CZF$? Is it the case that $\QLogic(\CZF) = \IQC$?
\end{question}

Due to the failure of exponentiation (see \Cref{Subsection: A failure of exponentiation}) it is clear that our techniques above cannot directly be used to obtain the results of this article for $\CZF$. With the semantics for $\CZF$ that the authors are currently aware of, it seems not possible to obtain mimic models for $\CZF$.

The situation for $\IZF$ is slightly different as Friedman and Ščedrov (see \Cref{Theorem: Friedman Scedrov}) showed that $\IQC \subsetneq \QLogic(\IZF) \subsetneq \CQC$. A challenging open problem is to give a better description of the first-order logic of $\IZF$. 

\begin{question}
    What is the first-order logic of $\IZF$? For example, is it possible to give an axiomatisation of $\QLogic(\IZF)$ or a concrete class of Kripke models that characterise $\QLogic(\IZF)$?
\end{question}

A first step in this direction might be to determine the relative first-order logic of $\IZF$.

Moreover, our study also contributes to the analysis of the admissible rules of the theory $\IKP$: Knowing the logic of a theory is the first step in analysing its admissible rules. For example, as we have shown that $\Logic(\IKP) = \IPC$, it follows that any propositional rule that is admissible in $\IKP$ must be admissible in $\IPC$ as well. It remains to determine the lower bound.

\begin{question}
    What are the admissible rules of $\IKP$?
\end{question}

\bibliographystyle{plain}
\bibliography{bibliography}

\end{document}